\documentclass[aps, prd, onecolumn, amssymb, amsmath, amsfonts]{revtex4-2}
\usepackage{lineno}
\usepackage{amsmath,amssymb}
\usepackage[english]{babel}
\usepackage{mathrsfs}
\usepackage[utf8]{inputenc}
\usepackage{amsthm}
\usepackage{amsfonts}
\usepackage[T2A]{fontenc}
\usepackage{graphicx}
\usepackage{hhline}
\usepackage{array}
\usepackage{subcaption, gensymb}
\usepackage{rotating}
\usepackage{multirow}
\usepackage{url}
\usepackage[thinlines]{easytable}
\usepackage{caption} 

\def\fracs#1#2{{}^{#1}\hspace{-0.25em}/%
	\hspace{-0.15em}_{#2}^{\vphantom{1}}}

\newtheorem{utv}{Proposition}
\makeatletter
\gdef\th@remark{%
	\thm@headfont{\small\scshape}%
}
\makeatother
\theoremstyle{remark}
\newtheorem{rem}{Remark}

\begin{document}
\title{ Trajectories of light beams in a Kerr metric: the influence of
the rotation of an observer on the shadow of a black hole}
\author{ I.A. Bizyaev}
\affiliation{Ural Mathematical Center, Udmurt State University, ul. Universitetskaya 1, 426034 Izhevsk, Russia}
\email{ bizyaevtheory@gmail.com}

\begin{abstract}
 This paper investigates the trajectories
of light beams in a Kerr metric, which describes
the gravitational field in the neighborhood of a
rotating black hole. After reduction by cyclic coordinates,
this problem reduces to analysis of a Hamiltonian system
with two degrees of freedom. A bifurcation diagram is
constructed and a classification is made of the types of
trajectories of the system according to the values of first
integrals. Relations describing the boundary of the shadow
of the black hole are obtained for a stationary observer who
rotates with an arbitrary angular velocity about the axis of
rotation of the black hole.
\end{abstract}	
\maketitle
	\section*{Introduction}
	 An analysis of the trajectory of light beams for a Kerr metric has been carried out by Bardeen
\cite{Bardeen1973}. He analyzed how the rotation of a black
hole influences the form of its shadow against a brightly glowing
background which is far away from the event horizon.
	   Interestingly, it turned out that this shadow is
asymmetric with respect to the axis of rotation of the black
hole and, as pointed out in \cite{Thorne}, resembles the forms of
letter D.
	
	   A detailed discussion of subsequent studies of
the trajectories of light beams in a Kerr metric is presented
in \cite{PerlickTsupkoCalculating, Dokuchaev, Perlick2}. In those studies,
the shadow of the black hole is usually considered in
a locally nonrotating reference system which
is sufficiently far away from the event horizon.
	  An exception is Ref. \cite{Grenzebach},
in which an equation was obtained for the boundary of the
shadow of the black hole for a Carter observer for whom the
radial coordinate is larger than the maximal radial coordinate
of the spherical photon orbit.
	
	   This paper continues the investigation of
the geodesics of a Kerr metric using methods of
qualitative analysis.
	   Previously, a bifurcation diagram was constructed in
\cite{BizyaevMamaev} for timelike geodesics. It was shown
that there exist seven diffent regions of values of first
integrals which differ in the topological type of integral
submanifold.  This paper is concerned with the isotropic
geodesics of a Kerr metric.
 In Section \ref{diagrammBifurcation},  a bifurcation
 diagram is constructed and a classification is made
 of the types of trajectories of light beams according to
 the values of first integrals.  Previously, such a
 diagram was constructed in \cite{Gralla}.
      In Section \ref{section3}, relations are obtained for
      the boundary of the shadow of a black hole for a
      stationary observer who rotates with an arbitrary
      angular velocity about the axis of rotation of the
      black hole. In addition, a backward tracing of light
      beams is constructed and it is shown how the image for
      the observer changes as the event horizon is approached.

	   The growing interest in research on the shadow of a 
black hole is fueled
by images obtained recently of supermassive objects M87* in the center of
galaxy M87 and SgA* in the center of our galaxy. Examination
of these images reveals
a bright accretion substance reaching the event horizon and
radiating in its vicinity
	   In this paper, we will not consider the image due to
an accretion disk. Such an image is discussed in detail, for example,
in \cite{ Luminet,Thorne,Dokuchaev}.

	\clearpage
	\section{Equations of motion of light beams}
In the Boyer\,--\,Lindquist coordinates
$\boldsymbol{x}=(t, r, \theta, \varphi)$ the interval
(line element) for a Kerr metric $g_{ij}$
is represented as follows \cite{Boyer}:
\begin{equation}
\label{eqKN03}
\begin{gathered}
ds^2 =g_{ij}dx^idx^j= -\frac{\rho^2 \Delta(r) }{A}dt^2 + \frac{A}{\rho^2}\sin^2\theta(d\varphi - \omega dt)^2 + \rho^2\left( \frac{dr^2}{\Delta(r)} + d\theta^2 \right), \\
\rho^2 = r^2 + a^2\cos^2\theta, \quad \Delta(r) = r^2 - 2r + a^2, \\
A = (r^2 + a^2)^2 - a^2\Delta(r) \sin^2\theta, \quad \omega = \frac{2r a}{A},
\end{gathered}
\end{equation}
where the signature $(-, +,+,+)$ is chosen and the
coordinates $r$ and $t$ are measured in the following units:
$$
	\frac{G M}{c^2}, \quad \frac{G M}{c^3},
$$
where $G$ is the gravitational constant, $c$ is the
velocity of light, and $M$ is the mass of the black hole.
	
The dimensionless parameter $a$ is expressed in terms of
the angular momentum
of the black hole $J$ relative to the symmetry axis
as follows:
$$
a = \frac{Jc}{G M^2}.
$$

At large distances, i.e., at $r \gg 1$ the metric
\eqref{eqKN03} becomes a flat Minkowski metric:
$$
ds^2 = dt^2 - \rho^2\left( \frac{dr^2}{r^2 + a^2} + d\theta^2\right) - (r^2 + a^2)\sin^2\theta d\varphi^2=dt^2 - dx^2 - dy^2 -dz^2,
$$
where the Cartesian coordinates are given by the following
relations:
\begin{equation}
\label{eq_x}
x=\sqrt{r^2 + a^2}\sin\theta\cos\varphi, \quad y=\sqrt{r^2 + a^2}\sin\theta\sin\varphi, \quad z=r\cos\theta.
\end{equation}

It follows from relations \eqref{eq_x} that
the level surfaces of the {\it radial
coordinates} $r={\rm const}$ at $t={\rm const}$
are confocal spheroids in the three-dimensional space
$$
\frac{x^2 + y^2}{r^2 + a^2} + \frac{z^2}{r^2}=1.
$$
The variables $\theta\in(0,\pi)$ and $\varphi\in[0,2\pi)$
are {\it polar} and {\it azimuth angles}, respectively,
and the coordinate $t$ plays the role of the {\it time} of an
external observer at rest.

The surface defining the {\it event horizon} for the
metric \eqref{eqKN03} is represented as
\begin{equation}
\begin{gathered}
\label{Eq_02}
\mathcal{S}_h=\{ (t, r, \theta, \varphi) \ | \ r=r_+  \}, \\
r_{+}=1 + \sqrt{1 - a^2},
\end{gathered}
\end{equation}
where $r_+$ is the dominant root of the equation $\Delta(r)=0$.
According to \eqref{Eq_02}, for an event horizon to exist, the
value $r_+$ must be real. This leads to the condition
$a\leqslant 1$. Furthermore, without loss of generality
we will assume that $a \geqslant  0$.

In what follows, we will consider the trajectories of light
beams only on the ``outer'' side of the event horizon
$\mathcal{S}_h$, i.e., on the manifold
$$
\mathcal{N}^4=\{ (t, r, \theta, \varphi) \ | \ t\in(-\infty, +\infty), \ r\in(r_+, +\infty), \ \theta\in(0,\pi), \ \varphi\in[0, 2\pi)  \}.
$$
We note that in this case the inequality $\rho(r, \theta)>0$
is satisfied.

The equations of motion for geodesics can be represented
in the following Hamiltonian form:
\begin{equation}
\label{eqH}
\begin{gathered}
\frac{d x^i}{d\tau} = \frac{\partial H}{\partial p_i}, \quad \frac{dp_i}{d\tau} = - \frac{\partial H}{\partial x^i}, \quad i=1,\dots, 4, \\
H=\frac{1}{2}g^{ij}p_ip_j,
\end{gathered}
\end{equation}
where $\boldsymbol{p}=(p_t, p_r, p_\theta, p_\varphi)$ is the
four-momentum along the trajectory and $g^{ij}$ is the
matrix inverse to the metric defined by relation
\eqref{eqKN03}.

The metric $g_{ij}$ does not depend explicitly on time $t$
and the angle $\varphi$, and hence they are cyclic coordinates
for equations \eqref{eqH}. As a consequence, the corresponding
momenta remain unchanged:
$$
E=-p_t={\rm const}, \quad L = p_\varphi = {\rm const},
$$
and the Hamiltonian system with two degrees of freedom
decouples from the system \eqref{eqH}:
\begin{equation}
\label{eqRed}
\begin{gathered}
\frac{dp_r}{d\tau} = - \frac{\partial H}{\partial r}, \quad \frac{dp_\theta}{d\tau} = - \frac{\partial H}{\partial \theta}, \quad
\frac{dr}{d\tau} = \frac{\partial H}{\partial p_r}, \quad \frac{d\theta}{d\tau} = \frac{\partial H}{\partial p_\theta}. \\
\end{gathered}
\end{equation}
The trajectories of light beams lie on the zero level set
of the Hamiltonian of this system:
\begin{equation}
\begin{gathered}
\label{eq_h_level}
H=\frac{1}{2\rho^2}\left( \Delta(r) p_r^2 + p_\theta^2 \right) + U(r, \theta)=0, \\
U(r, \theta)= \frac{1}{2\Delta(r) \rho^2}\left( 4arEL +\frac{\Delta(r)L^2}{\sin^2\theta} - a^2 L^2  \right) - \left(\frac{1}{2} +  \frac{r(r^2 + a^2)}{\Delta(r)\rho^2}\right)E^2,
\end{gathered}
\end{equation}
where the function $U(r, \theta)$
 is the effective potential.

 According to the well-known solution to the system
 \eqref{eqRed}, the evolution of the remaining variables is
 defined from the equations
\begin{equation}
\label{BLc2}
\begin{gathered}
\frac{d\varphi}{d\tau}=\frac{a}{\Delta(r)\rho^2}\left(E(r^2 + a^2) - aL\right) - \frac{aE}{\rho^2} + \frac{L}{\rho^2 \sin^2\theta}, \\
\frac{dt}{d\tau}=\frac{r^2 + a^2}{\Delta(r)\rho^2}\left(E(r^2 + a^2) - aL\right) + \frac{a}{\rho^2}\big(L - aE\sin^2\theta\big).
\end{gathered}
\end{equation}
In addition to the Hamiltonian $H$,  the reduced system
\eqref{eqRed} has an additional Carter integral \cite{Carter}
\begin{equation}
\label{Int}
F=p_\theta^2 + \left(aE\sin\theta - \frac{L}{\sin\theta} \right)^2 - 2a^2 H \cos^2\theta.
\end{equation}
For a more detailed discussion of the physical meaning
of the Carter integral, see \cite{Felice}.
Thus, we obtain the following well-known result:
the Hamiltonian system \eqref{eqRed} is integrable in the Liouville\,--\,Arnold sense. 	

\section{Bifurcation diagram and classification of trajectories }
\label{diagrammBifurcation}
 Both integrals of the system \eqref{eqRed} are quadratic in the momenta, and hence this system can be
 integrated by the method of separation of variables
 (for details, see \cite{Eisenhart1934}). In this case,
 $r$ and $\theta$ are separating variables. Therefore,
 in order to reduce the problem to quadratures, we fix
 the common level set of the Carter integral:
\begin{equation}
\label{In2t}
F=Q + (L - aE)^2,
\end{equation}
where $Q$ is some constant.

To analyze the trajectories of light beams, it is convenient
to introduce the quantities
\begin{equation}
\label{lambda_eta}
\lambda =\frac{L}{E}, \quad \eta=\frac{Q}{E},
\end{equation}
then we express the momenta $p_r$ and $p_\theta$ from \eqref{eq_h_level} and \eqref{Int} 
using \eqref{In2t} and substitute them into the last two
equations of motion of the system \eqref{eqRed}. As a result,
we obtain equations for $r$ and $\theta$ in the following
form:
\begin{equation}
\label{BLc}
\begin{gathered}
\left(\frac{dr}{d\tau}\right)^2 = \frac{E^2}{\rho^4}R(r),\quad
\left(\frac{d\theta}{d\tau}\right)^2 = \frac{E^2}{\rho^4} \Theta(\theta), \\
R(r)=(r^2 + a^2  - a \lambda)^2 - (\eta + (\lambda - a)^2)\Delta(r),\\
\quad
\Theta(\theta)=\eta  + \cos^2\theta\left(a^2 - \frac{\lambda^2}{\sin^2\theta} \right).
\end{gathered}
\end{equation}
As can be seen, in order to integrate these equations
in explicit form, we need to rescale time as
\begin{equation}
\label{BLc0}
d\tau = \frac{\rho^2(r, \theta)}{E} d\sigma.
\end{equation}
Using the new time variable and relations \eqref{lambda_eta},
we rewrite the system \eqref{BLc2} as
\begin{equation}
\label{eq_dphi}
\frac{d\varphi}{d\sigma}=\frac{a}{\Delta(r)}\left(r^2 + a^2 - a\lambda\right) - a + \frac{\lambda}{ \sin^2\theta},
\end{equation}
\begin{equation}
\label{eq_dt}
\frac{dt}{d\sigma}=\frac{r^2 + a^2}{\Delta(r)}\left(r^2 + a^2 - a\lambda\right) + a\big(\lambda - a \sin^2\theta\big).
\end{equation}
We first analyze the trajectories of equations \eqref{BLc}.
For them it holds that
\begin{itemize}
	\item[---] {\it the region of possible motion} on the plane $\mathbb{R}^2=\{ (r,\theta) \ | \ r>r_+, \ \theta\in(0,\pi) \}$
	is defined by the relations
	\begin{equation}
	\label{eq_OVD}
	R(r)\geqslant 0, \quad \Theta(\theta)\geqslant 0;
	\end{equation}
	\item[---] the simple zeros $r_u$ and $\theta_u$
of the functions
	$$
	R(r_u)=0, \quad \Theta(\theta_u)=0
	$$
	define the {\it turning points} of the
variables $r$ and $\theta$, respectively.
	\item[---]  those zeros $r_c$ and $\theta_c$ which
are simultaneously the critical points
	\begin{equation}
	\label{eqRdR}
	R(r_c)=0, \quad \left.\frac{dR}{dr}\right|_{r=r_c}=0,
	\end{equation}
	\begin{equation}
	\label{eqThdTh}
	\Theta(\theta_c)=0, \quad \left.\frac{d\Theta}{d\theta}\right|_{\theta=\theta_c}=0,
	\end{equation}
	define the invariant manifolds $r=r_c={\rm const}$ and  $\theta=\theta_c={\rm const}$.
\end{itemize}	
\begin{figure}[!ht]
\begin{center}
\includegraphics[scale=1.0]{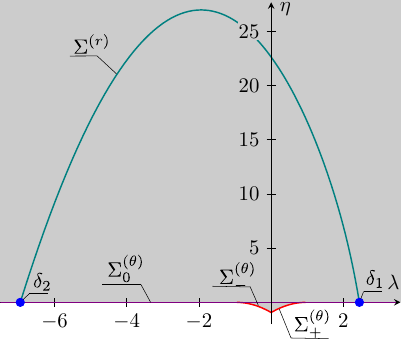}
\end{center}
\caption{Bifurcation diagram for the fixed $a=0.97$.
Gray denotes the values
of the integrals for which inequalities \eqref{eq_OVD} are satisfied. }
\label{fig2}
\end{figure}
Thus, investigation of the trajectories of light beams
reduces to analysis of the behavior of the functions $R(r)$
and $\Theta(\theta)$ depending on the values of the first
integrals $\lambda$ and $\eta$.

Solving relations \eqref{eqRdR} for $\lambda$, $\eta$ and
choosing those values
for which inequalities \eqref{eq_OVD} hold, we obtain a
parametrically given
curve on the plane of first integrals
$$
\begin{gathered}
\Sigma^{(r)} = \left\{ (\lambda, \eta) \ | \ \lambda =\frac{(1 + r_c)a^2 + r_c^2(r_c - 3)}{a(1-r_c)},  \ \eta=r_c^3\frac{ 4a^2 - r_c(r_c-3)^2 }{a^2(r_c - 1)^2}, r_c\in[r_1, r_2] \right\}, \\
r_1 = 2 + 2\cos\left(  \frac{2}{3}\arccos(-a)\right), \
r_2 = 2 + 2\cos\left(  \frac{2}{3}\arccos(a)\right)
\end{gathered}
$$
Solving relations \eqref{eqThdTh} for the values
of the first integrals, we obtain the straight line
$
\Sigma_0^{(\theta)}= \{ (\lambda, \eta) \ | \ \eta=0 \}
$
for which $\theta_c=\dfrac{\pi}{2}$,
and two curves
$$
\Sigma_{\pm}^{(\theta)}= \left\{ (\lambda, \eta) \ | \ \lambda = \pm a\sin^2\theta_c, \eta = -a^2\cos^4 \theta_c, \theta_c\in(0, \pi)  \right\},
$$
where the upper sign corresponds to $\Sigma_{+}^{(\theta)}$,
and the lower sign, to $\Sigma_{-}^{(\theta)}$.

The curve $\Sigma^{(r)}$ intersects with the straight line
$\Sigma_0^{(\theta)}$ at two points, which we denote by
$\delta_1$ and  $ \delta_2$.
They correspond to the fixed points of the reduced system
\eqref{eqRed} lying in the equatorial plane
$\theta=\dfrac{\pi}{2}$
and to the values of the radial coordinate, $r_1$ and
$r_2$, for $\delta_1$ and $\delta_2$, respectively.

\begin{rem}
Equations \eqref{eqRdR} have another solution, namely,
$$
\eta=-\frac{r_c^4}{a^4}, \quad \lambda = \frac{a^2 + r_c^2}{a},
$$
but it must be eliminated because in this case
$$
\Theta(\theta)=-\frac{(r_c^2 + a^2\cos^2\theta )^2}{a^2\sin^2 \theta}<0.
$$
\end{rem}

\begin{figure}[!ht]
\begin{center}
\includegraphics[scale=1.0]{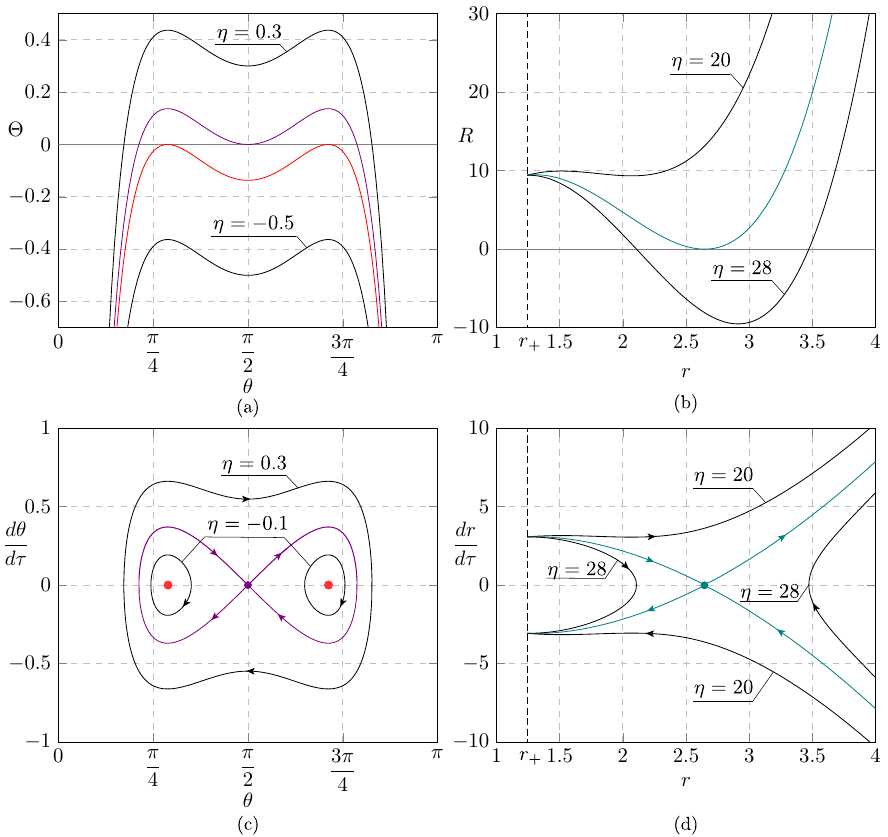}
\end{center}
\caption{For fixed $\lambda=-0.6$ and
$a=0.97$  (a) the function $\Theta(\theta)$, (b) the function
$R(r)$, and (с)-(d) the phase portraits, versus $\eta$.
The colored curves in these figures are plotted for the
values of the first integrals
lying on the bifurcation curves shown in the same colors in
Fig. \ref{fig2}.
}
\label{fig1}
\end{figure}

Typical curves found on the plane of first integrals are
shown in Fig. \ref{fig2}. As is well known from the analysis
of the function of a real variable, changes in the region of possible motion
\eqref{eq_OVD} and in the number of the turning points
$r_u$ or $\theta_u$ occur for the values of the first
integrals lying on these bifurcation curves
\cite{Bolsinov}.

\begin{rem}
The bifurcation curves thus found describe local
bifurcations. No nonlocal bifurcations due to changes in the
behavior of the function $R(r)$ arise on the boundaries of
the interval $(r_+, +\infty)$. Nonlocal bifurcations arise
for the motion of a material point (i.e., $H=\fracs{1}{2}$)
and are examined in detail in \cite{BizyaevMamaev}.
\end{rem}

\begin{rem}
The bifurcation diagram shows the point $(0, -a^2)$ at
which the curves $\Sigma_-^{(\theta)}$ and $\Sigma_+^{(\theta)}$ merge.
This point corresponds to the trajectory that lies entirely
on the symmetry axis.
\end{rem}

Figure \ref{fig1} shows examples of the functions
$\Theta(\theta)$ and $R(r)$ with different numbers of
turning points. Note that, if $\eta<0$, then the trajectories
do not cross the equatorial plane. Moreover, the inequality
$\dfrac{dt}{d\sigma}>0$ always holds in the region of possible
motion, and so we will not analyze the dependence
$t(\sigma)$ in detail.

Critical solutions for which $r(\sigma)=r_c$
are called spherical orbits. They are examined in detail
in \cite{Teo}. For the points $\delta_1$ and $\delta_2$
they reduce to flat trajectories which are circles.
Furthermore, for $\delta_2$ the direction of rotation
does not coincide with the direction of rotation of the
black hole since in this case $\dfrac{d\varphi}{d\sigma}<0$,
whereas for $\delta_1$ the direction of rotation does coincide
with the direction of rotation of the black hole since
$\dfrac{d\varphi}{d\sigma}>0$.

Let us consider in more detail the behavior of the function
$r(\sigma)$ depending on the values of the first integrals.
 According to the bifurcation diagram, two cases should be distinguished.
\begin{itemize}
\item[(i)]  If the point $(\lambda, \eta)$ lies below the curve
$\Sigma^{(r)}$, then the turning point $r_u$ is absent and
all trajectories, as they continue in $\tau$ in one direction, go to infinity 
$(r\to +\infty)$, and as they continue in $\tau$ in the opposite direction, 
they approach the horizon. In this paper, we call such trajectories
the ``horizon/infinity''.
\item[(ii)] If the point $(\lambda, \eta)$ lies above the
curve $\Sigma^{(r)}$, then there are two turning points $r_u$.
One turning point corresponds to a trajectory of the type
``horizon/horizon'', and the other, to a trajectory of the type
``infinity/infinity''.
\end{itemize}
In each of these regions the solution is given in terms of
elliptic functions (for details, see, e.g., \cite{Gralla, Gralla2}.
For the values of the integrals lying on the curve
$\Sigma^{(r)}$ in the phase portrait of the system
(see Fig. \ref{fig1}d) there are a saddle fixed point
and asymptotic trajectories adjacent to it ({\it separatrices}).

For the separatriсes the solution $r(\sigma)$ can be
expressed in elementary functions. In order to obtain it, we
substitute into the function $R(r)$ the values of the
integrals $\lambda$ and $\eta$ on the curve $\Sigma^{(r)}$
and represent this function as
\begin{equation}
\label{ss1}
\begin{gathered}
R(r)=\pm(r-r_c)\sqrt{r^2 + 2r_cr - 3r_c^2 +Z^2}, \\ Z=\frac{2\sqrt{r_c}}{1-r_c}\big(  3 r_c - (a^2 + 3 r_c^2) + r_c^3 \big)^{\fracs{1}{2}}.
\end{gathered}
\end{equation}
Integrating the first of equations \eqref{BLc}
using the new time variable \eqref{BLc0}, we find
an expression for the separatrix which, for the initial value
of the radial coordinate $r(0)>r_c$, approaches
asymptotically the solution $r=r_c$:
\begin{equation}
\label{eq_r_sigma}
\begin{aligned}
r(\sigma)=r_c + Z^2 \big(r(0) - r_c\big)\Big( &\big(Z^2 + 2r_c(r(0) - r_c)\big) \cosh (Z\sigma) + \Big.\\ \Big.
+&Z\sqrt{Z^2 + \big(r(0) + r_c^2\big) - 4r_c^2}\sinh(Z\sigma) - 2 r_c\big(r(0) - r_c\big) \Big)^{-1}.
\end{aligned}
\end{equation}
\begin{figure}[!ht]
\begin{center}
\includegraphics[scale=1.1]{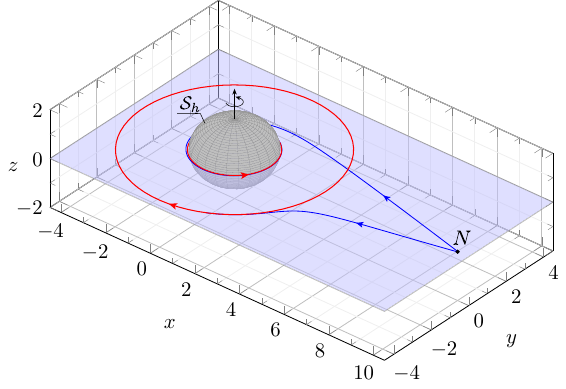}
\end{center}
\caption{Typical circular trajectories (red) and separatrices
(light blue) for the fixed parameter $a=0.97$ and the initial
conditions $r(0)=10$ and $\varphi(0)=0$, which correspond to
point $N$.}
\label{fig3}
\end{figure}

Figure \ref{fig3} shows separatrices lying in the equatorial
plane $\theta=\dfrac{\pi}{2}$. For them, the dependence
$\varphi(\sigma)$ was obtained after integration of
equation \eqref{eq_dphi} taking \eqref{eq_r_sigma} into
account. The flat trajectories with the initial conditions from
point $N$ which lie between the separatrices plunge into the event horizon, and the other
trajectories go to infinity. Note that an asymmetry arises
in these plunging trajectories with respect to the axis
of rotation of the black hole, which is caused
by the positions of the unstable circular orbits.
As the parameter $a$ increases, this asymmetry becomes more pronounced
because the distance between the unstable circular orbits
increases.

\section{Stationary reference systems and the shadow of the
black hole}
\label{section3}
Consider reference systems which outside the horizon
$\mathcal{S}_h$ rotate about the axis of rotation of the
black hole with constant angular velocity $\Omega$.
For them, the space-time remains stationary, and hence we will
call them {\it stationary reference systems} \cite{Semerak}.
Suppose that at the origin of such a reference system
there is an observer for whom the Boyer\,--\,Lindquist
coordinates are given by the following relations:
$$
r=r_0={\rm const}, \ \theta=\theta_0={\rm const}, \ \varphi=\Omega t + \varphi_0.
$$
For such a stationary observer the four-velocity has the following form:
$$
\begin{gathered}
\boldsymbol{u}=u_0\left(\frac{\partial}{\partial t} + \omega_0 \frac{\partial}{\partial \varphi}\right), \quad
u_0 =\frac{dt}{d\tau} = \rho_0\Big( \Delta(r_0) - a^2\sin^2\theta_0 + \Omega(4ar_0 - A_0 \Omega) \sin^2 \theta_0 \Big)^{-\fracs{1}{2}},
\end{gathered}
$$
where  $\tau$ is the proper time and the following
notation is used:
$$
\rho_0^2 = r_0^2 + a^2\cos^2\theta_0, \quad
A_0 = (r_0^2 + a^2)^2 - a^2\Delta(r_0) \sin^2\theta_0,\quad  \omega_0=\frac{2r_0 a}{A_0}.
$$
In order that the observer's trajectory remain
timelike, the radicand for $u_0$
must take positive values. From this we obtain the following
range of admissible angular velocities:
$$
\Omega\in(\Omega_{ -}, \Omega_{+}), \quad  \Omega_{\pm} = \frac{1}{r_0^2 + a^2 \pm a \sqrt{\Delta(r_0)} \sin \theta_0} \left(a \pm  \dfrac{\sqrt{\Delta(r_0)}}{\sin\theta_0}\right),
$$
where the lower sign corresponds to $\Omega_{-}$ and the upper sign, to
$\Omega_{+}$.

Examples of observers which are usually given for a Kerr
metric are summarized in Table \ref{tab1} (for details, see, e.g., \cite{Semerak}). Recall
that a static observer can be defined only outside the
ergosphere.

\begin{table}[ht]	
\caption{ Examples of observers in the case of a Kerr metric.}
\begin{center}
\begin{tabular}{ | p{150pt} |  c | p{185pt}| }
\hline
name  &   angular velocity $\Omega$   & time component of velocity $u_0$ \\ \hline
zero angular momentum observer or ZAMO & $\omega_0$ & $ \dfrac{1}{\sqrt{\Delta(r_0)}}\left(r_0^2 + a^2 + \dfrac{2r_0 a^2\sin^2 \theta_0}{\rho_0^2} \right)^{\fracs{1}{2}}$  \\ \hline
 static observers & 0  & $\quad \quad \quad    \dfrac{\rho_0}{\sqrt{r^2_0 - 2 r_0 + a^2\cos^2\theta_0}}$ \rule[-2ex]{0pt}{6ex} \\ \hline
 Carter observers & $\dfrac{a}{r_0^2 + a^2}$  & $\quad \quad \quad \quad \quad \dfrac{r_0^2 + a^2}{\rho_0 \sqrt{\Delta(r_0)}}$ \rule[-2ex]{0pt}{6ex} \\ \hline
\end{tabular}
\label{tab1}
\end{center}
\end{table}

The basis vectors of the orthonormal tetrad which are related 
to the stationary observer have the following form (see Fig. \ref{fig4}a):
$$
\begin{gathered}
\boldsymbol{e}_{(t)}=\boldsymbol{u}, \quad
\boldsymbol{e}_{(r)} = -\frac{\sqrt{\Delta(r_0)}}{\rho_0}\frac{\partial}{\partial r}, \quad
\boldsymbol{e}_{(\theta)} = -\frac{1}{\rho_0}\frac{\partial }{\partial \theta}, \\
\boldsymbol{e}_{(\varphi)}  = \frac{u_0}{\sin\theta_0\sqrt{\Delta(r_0)}} \left[ \frac{A_0\sin^2\theta_0}{\rho_0^2}(\Omega - \omega_0) \frac{\partial}{\partial t}  +
 \left(1 - 2 r_0 \frac{1 - a \Omega \sin^2 \theta_0}{\rho_0^2} \right) \frac{\partial}{\partial \varphi} \right].
\end{gathered}
$$

\begin{figure}[!ht]
\begin{center}
\includegraphics[scale=1.8]{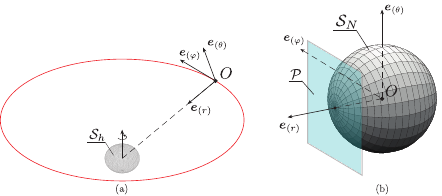}
\end{center}
\caption{ (a) Schematic diagram of the horizon
$\mathcal{S}_h$ and (b) schematic diagram of the
sphere $\mathcal{S}_N$ and the plane $\mathcal{P}$
relative to the spatial components of the vectors $\boldsymbol{e}_{(r)}$,
$\boldsymbol{e}_{(\theta)}$ and $\boldsymbol{e}_{(\varphi)}$.}
\label{fig4}
\end{figure}

We describe what a stationary observer must see in the
neighborhood of the black hole.  Assume that a point source
of light beams is at the origin and that these beams evolve
backward in time $\tau$. The tangent vector to such an
arbitrary light beam is given by the relation
\begin{equation}
\label{eq_w}
\begin{gathered}
\boldsymbol{w} = W(-\boldsymbol{e}_{(t)} + N_1 \boldsymbol{e}_{(\theta)} + N_2\boldsymbol{e}_{(\varphi)} + N_3 \boldsymbol{e}_{(r)} ), \\
\end{gathered}
\end{equation}
where the components of the radius vector
$\boldsymbol{N}=(N_1, N_2, N_3)$ lie on the sphere
$$
\mathcal{S}_N = \left\{ N_1= \sin \alpha \cos\beta, \ N_2 =  \sin\alpha \sin\beta, \  N_3=\cos\alpha \ | \  \alpha\in(0, \pi), \ \beta\in[0, 2\pi) \right\}.
$$

The trajectories with the initial conditions \eqref{eq_w}
which evolve onto the event horizon will form the shadow of
the black hole \cite{Bardeen1973, PerlickTsupkoCalculating}
against the background of light beams from remote sources.
In order to describe the boundary of this shadow, we fix
the coordinates $r_0$ and $\theta_0$ of the observer
and his angular velocity $\Omega$. Next, we define three
functions which depend on $r_c$ and
have the following form:
 $$
 \begin{gathered}
 B_\alpha(r_c) = \arccos\left[\frac{(r_0-r_c)\sqrt{r_0^2 + 2r_cr_0 - 3r_c^2 +Z^2}}{\rho_0u_0\big(1 - \lambda(r_c)\Omega\big)\sqrt{\Delta(r_0)}} \right], \
 \Theta_*(r_c)=\eta(r_c)  + \cos^2\theta_0\left(a^2 - \frac{\lambda^2(r_c)}{\sin^2\theta_0} \right), \\
 B_\beta(r_c) =  \arctan\left[\frac{u_0\big( \rho_0^2(\lambda(r_c) - \Omega(r_0^2 + a^2)\sin^2\theta_0 ) +2r_0(1-a\Omega\sin^2\theta_0)(a\sin^2\theta_0 - \lambda(r_c))  \big)}{\rho_0\sin\theta_0\sqrt{\Delta(r_0)} \sqrt{\Theta_*(r_c)}}\right],
 \end{gathered}
 $$
where $Z$ is given by \eqref{ss1}, and the functions
$\lambda(r_c)$ and $\eta(r_c)$ are defined by the
values of the integrals for the curve $\Sigma^{(r)}$.

\begin{figure}[!ht]
\begin{center}
\includegraphics[scale=1.2]{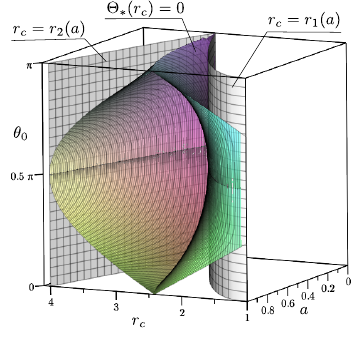}
\end{center}
\caption{The position of the surface  $\Theta_*(r_c)=0$
relative to the surfaces $r_c=r_1(a)$ and $r_c=r_2(a)$.}
\label{fig3d}
\end{figure}
The function $\Theta_*(r_c)$ always has two zeros,
$r^*_1$ and $r^*_2$, which satisfy the inequalities
(see Fig. \ref{fig3d})
$$
r_2 \geqslant r_2^*>r_1^*\geqslant r_1 .
$$
In this case, the following statement holds.
\begin{utv}
\label{utv1}
The curve describing the boundary of the shadow of the black
hole on the sphere $\mathcal{S}_N$ is given as follows:
\begin{itemize}
\item [1)] if $r_0>r_2^*$, then
\begin{equation}
\begin{gathered}
\label{eq_ab}
\mathcal{C}=  \big\{ \alpha = B_\alpha(r_c),   \beta=\pi +  B_\beta(r_c) \big\} \cup \big\{ \alpha = B_\alpha(r_c),   \beta= 2\pi - B_\beta(r_c) \big\},
 \end{gathered}
\end{equation}
where $r_c\in[r_1^*, r_2^*];$
\item [2)] if $ r_1^* > r_0$, then
\begin{equation}
\label{eq_ab2}
\begin{gathered}
\mathcal{C}=\big\{ \alpha = \pi - B_\alpha(r_c),   \beta=\pi +  B_\beta(r_c) \big\} \cup \big\{ \alpha = \pi - B_\alpha(r_c),   \beta= 2\pi - B_\beta(r_c) \big\},
\end{gathered}
\end{equation}
where $r_c\in[r^*_1, r^*_2];$
\item [3)] if $r^*_2 \geqslant r_0 \geqslant r^*_1$, then
the curve $\mathcal{C}$ is given by relations \eqref{eq_ab}
and \eqref{eq_ab2}, in which $r_c\in[r_0, r^*_2]$ and
$r_c\in[r_1^*, r_0)$, respectively.
\end{itemize}
\end{utv}
\begin{proof}
We first express the angles $\alpha$ and $\beta$ in terms of $\dfrac{d\theta}{d\tau}$, $\dfrac{dr}{d\tau}$, $L$ and $E$.
To do so, we represent the tangent vector \eqref{eq_w} to the light beam as
\begin{equation}
\label{eq_w2}
\boldsymbol{w}=\frac{dr}{d\tau}\frac{\partial}{\partial r} + \frac{d\theta}{d\tau}\frac{\partial}{\partial \theta} + \frac{d\varphi}{d\tau}\frac{\partial}{\partial \varphi} + \frac{dt}{d\tau}\frac{\partial}{\partial t}.
\end{equation}
 Setting relations \eqref{eq_w} and
 \eqref{eq_w2} equal to each other, we find
\begin{equation}
\label{eq_00}
\begin{gathered}
\frac{dr}{d\tau}=-\frac{W\sqrt{\Delta(r_0)}}{\rho_0}\cos\alpha, \quad \frac{dt}{d\tau}=-Wu_0\left( 1 - \frac{\Omega A_0 - 2ar_0}{\rho_0^2\sqrt{\Delta(r_0)}}\sin \alpha \sin \beta \sin\theta_0 \right), \\
\frac{d\theta}{d\tau}=-\frac{W}{\rho_0}\sin \alpha \cos\beta, \quad \frac{d\varphi}{d\tau}=-Wu_0\left(\Omega + \frac{a(a - 2 r_0\Omega)\sin^2\theta_0 - \Delta(r_0)}{\rho^2_0\sqrt{\Delta(r_0)}\sin\theta_0}\sin\alpha \sin\beta \right).
\end{gathered}
\end{equation}
We rewrite the relations for the integrals as
$$
\begin{gathered}
E=\frac{2ar_0\sin^2\theta_0}{\rho_0^2}\frac{d\varphi}{d\tau} +\left(  1 - \frac{2r_0}{\rho_0^2}  \right)\frac{dt}{d\tau}, \quad
L=\frac{A_0\sin^2\theta}{\rho_0^2}\frac{d\varphi}{d\tau}  - \frac{2ar_0\sin^2\theta_0}{\rho_0^2}\frac{dt}{d\tau}.
\end{gathered}
$$
Let us multiply the second equation by $\Omega$. Subtracting
the first equation from it, we obtain an equation that does
not depend
explicitly on the angles $\alpha$ and $\beta$. From this equation we find
$$
W=u_0(\Omega L - E).
$$
With \eqref{lambda_eta} taken into account, relations \eqref{eq_00}
 yield
\begin{equation}
\label{eq_ang5}
\cos \alpha = \frac{\rho_0}{u_0 E(1 - \Omega \lambda)\sqrt{\Delta(r_0)}}\frac{dr}{d\tau}, \quad
\sin\alpha \cos \beta = \frac{\rho_0}{u_0(1 - \Omega \lambda)} \frac{d\theta}{d\tau}.
\end{equation}

With the initial coordinates from the position of the observer, the noncritical trajectories 
of light beams traveling back in time  go to infinity or reach the event horizon. 
They are separated from each other by two unstable separatrices (in forward time)
to a spherical trajectory (for details, see the phase portrait in Fig. \ref{fig1}d). These
separatrices correspond to the sought-for values of the angles $\alpha$ and $\beta$.
As shown previously in Section \ref{diagrammBifurcation}, for the
separatrices the values of the integrals lie on the
bifurcation curve $\Sigma^{(r)}$,  and hence $\lambda$ and $\eta$ are given by the following
relations:
\begin{equation}
\label{0eq}
\lambda(r_c) =\frac{(1 + r_c)a^2 + r_c^2(r_c - 3)}{a(1-r_c)},  \ \eta(r_c)=r_c^3\frac{ 4a^2 - r_c(r_c-3)^2 }{a^2(r_c - 1)^2}.
\end{equation}

Recall that the region of possible motion is defined by
inequalities \eqref{eq_OVD}. In the case \eqref{0eq}
it follows from the second inequality that $r_c\in[r_1^*, r_2^*]$.
Moreover, if $r_0>r_2^*$ then
\begin{equation}
\label{eqProf1}
\frac{dr}{d\tau} = \frac{E}{\rho_0^2}\sqrt{R(r_0)},\quad
\frac{d\theta}{d\tau} = \pm\frac{E}{\rho_0^2} \sqrt{ \Theta(\theta_0)},
\end{equation}
in the first relation, $\pm$ is absent because in this case the inequality 
$\dfrac{dr}{d\tau}>0$ always holds for an unstable separatrix.
For $r_1^*>r_0$ we have
\begin{equation}
\label{eqProf2}
\frac{dr}{d\tau} = -\frac{E}{\rho_0^2}\sqrt{R(r_0)},\quad
\frac{d\theta}{d\tau} = \pm\frac{E}{\rho_0^2} \sqrt{ \Theta(\theta_0)},
\end{equation}
in this case the opposite inequality always holds for an unstable separatrix:
$\dfrac{dr}{d\tau}<0$ (see Fig. \ref{fig1}d).

If $r^*_2 \geqslant r_0 \geqslant r^*_1$, then for the
separatrices in \eqref{eqProf1} we need to choose $r_c \in [r_0, r_2^*]$, and  
in \eqref{eqProf2} we need to choose $r_c \in [r_1^*, r_0)$.
 Finally, we substitute \eqref{eqProf1} and \eqref{eqProf2}
 into \eqref{eq_ang5} and, after some transformations,
 we obtain the sought-for relations \eqref{eq_ab} and \eqref{eq_ab2}.
\end{proof}

\begin{rem}
We note that the functions $B_\alpha(r_c)$ and $B_\beta(r_c)$
simplify greatly for Carter observers and
hence the values of the angles $\alpha$ and $\beta$
for the boundary of the shadow of the black hole for $r_0>r_1$
can be reduced to the relations
$$
\sin \alpha = \frac{2\sqrt{\Delta(r_c)}\sqrt{\Delta(r_0)}}{\Delta(r_c) + \Delta(r_0) - r_c^{-1}(r_0 - r_c)^2}, \quad \sin\beta = \frac{a^2\cos^2\theta_0(r_c-1) - r_c(r_c^2 - 3r_c+2a^2)}{2ar_c\sqrt{\Delta(r_c)}}.
$$
Previously, these relations were obtained in \cite{PerlickTsupkoCalculating}.
\end{rem}

To visualize the curve  $\mathcal{C}$ found above, we perform a stereographic projection 
of the sphere $\mathcal{S}_N$ onto the plane $\mathcal{P}$. The position of the plane is 
shown in Fig. \ref{fig4}b. The Cartesian coordinates on this plane are given by the relations
$$
X=2 \tan \frac{\alpha}{2}\sin\beta, \quad Y = 2 \tan \frac{\alpha}{2} \cos\beta.
$$
Figure \ref{fig5}a  shows the curve $\mathcal{C}$ on the plane $\mathcal{P}$ plotted 
against 
the angle $\theta_0$. As can be seen, the shadow of the black hole is not 
axisymmetric, and this asymmetry is seen most clearly for the equatorial plane.
Figure \ref{fig5}b shows the shadow of the black hole plotted 
against the angular velocity
$\Omega$ in the equatorial plane.
Also, Fig. \ref{figStatic} shows how the boundary of the shadow
of the black hole changes depending on the radial coordinate of a static observer and 
a zero angular momentum observer.

\begin{figure}[!ht]
\begin{center}
\includegraphics[scale=1.0]{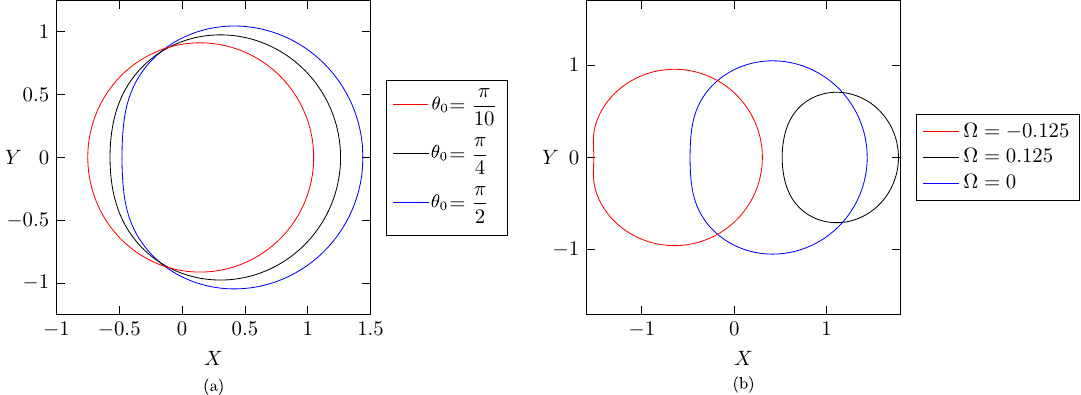}
\end{center}
\caption{ The curve $\mathcal{C}$ for the fixed $a=0.98$, $r_0=5$: (a) different 
$\theta_c$ and the fixed $\Omega=0$, (b) different $\Omega$ and the fixed 
$\theta_0=\displaystyle\frac{\pi}{2}$.}
\label{fig5}
\end{figure}

\begin{figure}[!ht]
\begin{center}
\includegraphics[scale=1.0]{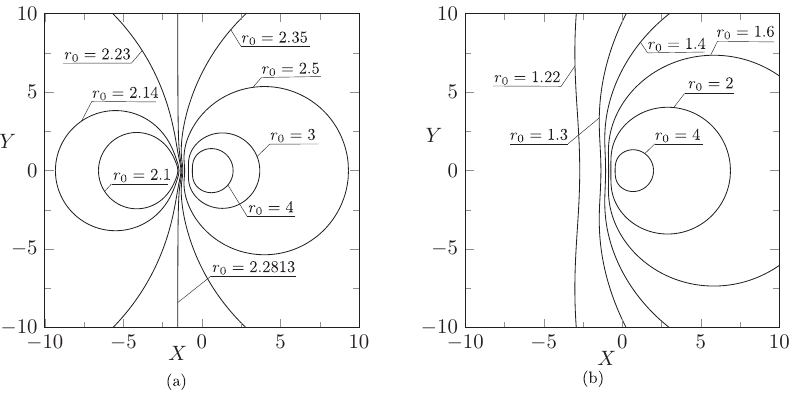}
\end{center}
\caption{The curve $\mathcal{C}$ for the fixed $a=0.98$, $\theta_0=\displaystyle\frac{\pi}{2}$ 
and different $r_0$: (a) for a static observer $\Omega=0$, (b) for a zero angular momentum 
observer  $\Omega=\omega_0$.}
\label{figStatic}
\end{figure}

\begin{figure}[!ht]
\begin{center}
\includegraphics[scale=1.1]{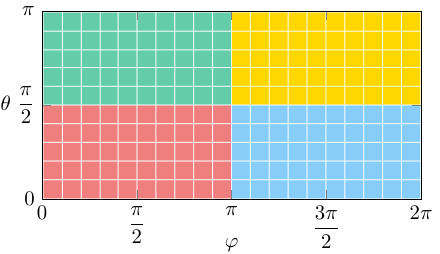}
\end{center}
\caption{Diagram of the celestial sphere on the plane of angles $(\varphi, \theta)$.}
\label{fig8}
\end{figure}

To visualize the image that a stationary observer must see, we perform
a backward tracing of light beams  \cite{Muller, Luminet, OSIRIS}. Below we describe briefly 
its construction.
\begin{itemize}
\item [---] Define the celestial sphere as a source of light
beams for which the value of the radial coordinate is $r \gg 1$.
Since $a \leqslant 1$, we can neglect in \eqref{eq_x} the
contribution of the parameter $a$ for the celestial
sphere. Single out four regions on the celestial sphere and show them in different colors 
as depicted in Fig. \ref{fig8}.
\item [---]  For the fixed initial conditions, from the position of the observer $r_0$, $\theta_0$ and $\varphi(0)=0$, but with different
values of the angles $\alpha$ and $\beta$, we will numerically integrate the system  
\eqref{eqRed} and the first equation in the system \eqref{BLc2} in backward time, i.e., 
after the time reversal $\tau \to -\tau$.
\item [---]  On the plane $\mathcal{P}$, the trajectories
of light beams which do not reach the celestial sphere will be shown in black,
and the other trajectories will be shown in colors depending on
what part of the celestial sphere the trajectory will reach.
\end{itemize}
The image on the plane $\mathcal{P}$ after the backward tracing of light beams is depicted in 
Figs. \ref{fig6} and \ref{ray}. In these figures one can clearly see a curvature of 
the trajectories of light beams due to the gravitational field. Moreover, 
as the radial coordinate decreases, the black region on the plane  $\mathcal{P}$ 
increases. This region corresponds to the shadow of the black hole.
The boundary of this region is in agreement with the relations obtained above in 
Statement \ref{utv1} (see Fig. \ref{figStatic}a).

\begin{figure}[!ht]
\begin{center}
\includegraphics[scale=0.9]{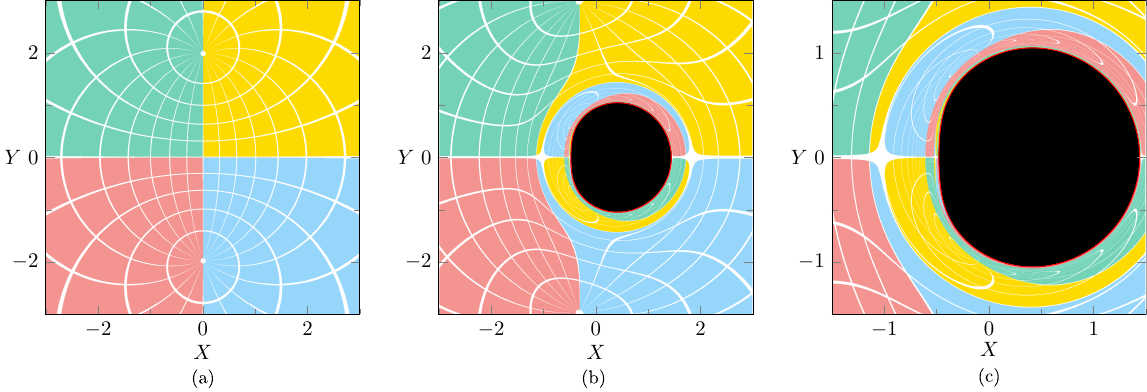}
\end{center}
\caption{ (a) The celestial sphere on the plane $\mathcal{P}$  in the plane space; 
(b) image on the plane $\mathcal{P}$ in the neighborhood of the black hole after the backward tracing of light beams and (c) enlarged fragment.  The observer is in the equatorial plane
$\theta_0=\dfrac{\pi}{2}$ with $r_0=5$  and angular velocity
$\Omega=0$. The celestial sphere is at distance $r=1000$.
Red denotes the shadow's boundary obtained from relations \eqref{eq_ab}.
 }
\label{fig6}
\end{figure}

\begin{figure}[!ht]
\begin{center}
\includegraphics[scale=1.0]{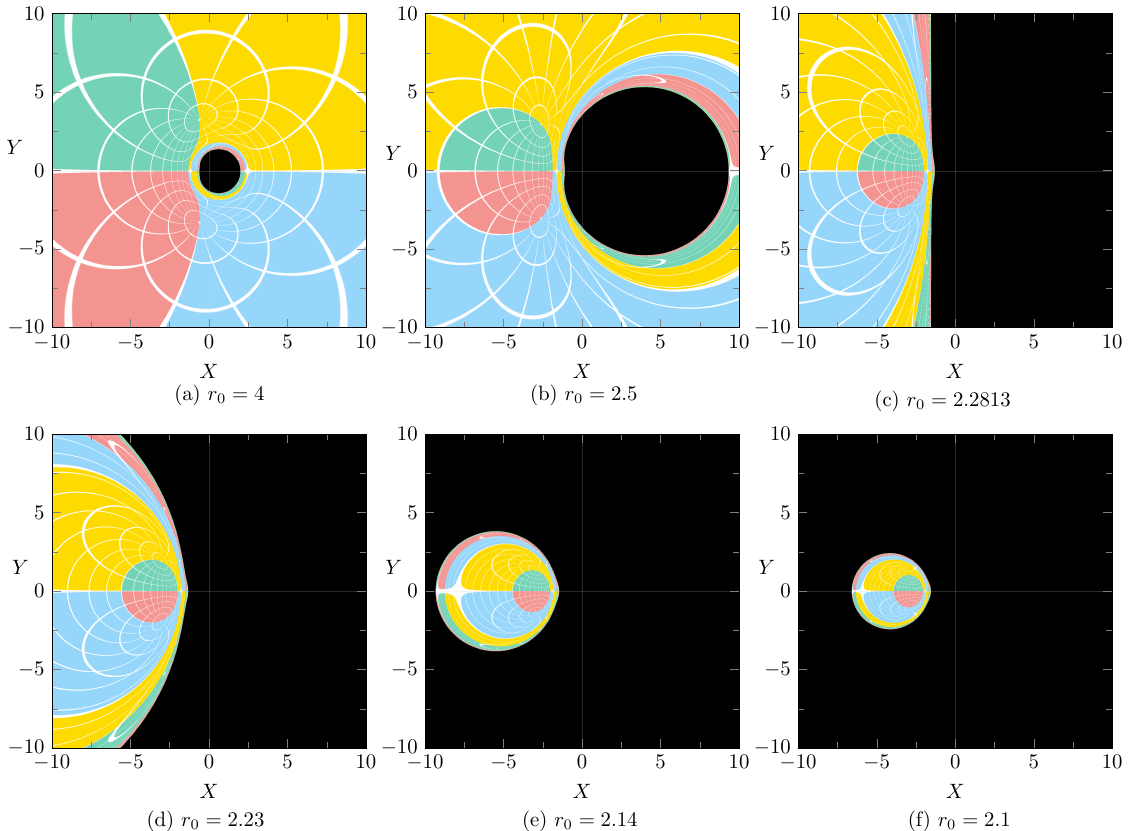}
\end{center}
\caption{The image on the plane $\mathcal{P}$ in the neighborhood of the black hole after the backward tracing of light beams, plotted against the radial coordinate $r_0$ of the stationary
observer $\Omega=0$
and fixed $\theta_0=\displaystyle\frac{\pi}{2}$, $a=0.98$.
The celestial sphere is at distance $r=1000$. }
\label{ray}
\end{figure}

A classification of trajectories of light beams was
carried out within the framework of the state assignment of the Ministry of Science and 
Higher Education of Russia (FEWS-2024-0007).  An analysis of the
boundary of the shadow of the black hole and the tracing of light beams was performed at 
the Ural Mathematical Center (Agreement No. 075-02-2024-1445).

\end{document}